\documentclass[12pt]{iopart}

\usepackage{iopams}

\usepackage{graphicx}
\usepackage{amssymb}
\usepackage{amsthm}
\usepackage{latexsym}

\theoremstyle{plain}
\newtheorem{theorem}{Theorem}[section]

\newtheorem{lemma}[theorem]{Lemma}
\newtheorem{Problem}[theorem]{Problem}

\theoremstyle{definition}
\newtheorem{definition}[theorem]{Definition}

\def\clo#1{\overline{#1}}
\def\text#1{\mbox{#1}}

\newcommand{\ra}{\rangle}
\newcommand{\la}{\langle}

\eqnobysec

\usepackage[letterpaper,margin=1in]{geometry}

\begin{document}

\title{On the multi-frequency inverse source problem in heterogeneous media}

\author{S Acosta$^{1,2}$, S Chow$^2$, J Taylor$^2$ and V Villamizar$^2$}
\address{$^1$Computational and Applied Mathematics, Rice University, Houston, TX 77005, USA}
\address{$^2$Department of Mathematics, Brigham Young University, Provo, UT 84602, USA}
\ead{sebastian.acosta@rice.edu}
\begin{abstract}
The inverse source problem where an unknown source is to be identified from the knowledge of its radiated wave is studied. The focus is placed on the effect that multi-frequency data has on establishing uniqueness. In particular, it is shown that data obtained from finitely many frequencies is not sufficient. On the other hand, if the frequency varies within a set with an accumulation point, then the source is determined uniquely, even in the presence of highly heterogeneous media. In addition, an algorithm for the reconstruction of the source using multi-frequency data is proposed. The algorithm, based on a subspace projection method, approximates the minimum-norm solution given the available multi-frequency measurements. A few numerical examples are presented.
\end{abstract}

\submitto{\IP}


\section{Introduction} \label{Section:Intro}
We study the inverse source problem for the Helmholtz equation where an unknown source is to be identified from the knowledge of its radiated wave. This problem has particular applications in the antenna synthesis problem \cite{Balanis2005,AngKirKlein1991} as well as in medical imaging techniques such as electroencephalography (EEG) \cite{AHLM2009,AlKressSerr2009}, magnetoencephalography (MEG) \cite{HeRom1998,AmBaoFlem2002,AlbMonk2006}, photo- and thermo-acoustic tomography \cite{AZML2007,StefUhl2009,StefUhl2011}, and optical tomography \cite{ArrRev1999}. For instance, the recording of the brain's electrical activity (such as the origin, path and destination of electrical currents) may be used to relate external stimuli (visual, tactile, auditory) to neuronal connectivity and brain functionality \cite{JirMc2007,LuKaufman2003}. Hence, the reconstruction of current sources with spatial resolution is particularly crucial in this branch of scientific investigation. An excellent review of some of the latest mathematical developments for medical imaging is found in \cite{Amm2009,Amm2012}. Particular attention is given therein to imaging techniques that can provide insight or monitoring of physiological events and be employed as a research tool in cognitive science.

The direct source problem has been thoroughly investigated over the past centuries and a huge amount of information is available in the literature. In contrast, the inverse source problem has received much less attention. The reason for this is that the inverse source problem is inherently ill-posed in terms of uniqueness and stability. Due to the existence of non-radiating sources, infinitely many solutions can be obtained. Hence, it is impossible for the true source to be uniquely reconstructed from a set of measurements at a single frequency. This phenomenon holds true for problems governed by the Helmholtz (acoustics), Maxwell (electrodynamics) and Laplace (gravimetry and electrostatics) equations \cite{AlbMonk2006,Isakov1998,Isakov1990,HauKuhPott2005,DassFokKar2005,BleCoh1977,MarDevZio2000}.

In most mathematical works, one or more of the following assumptions are made in order to simplify the inverse source problem:
\begin{itemize}
\item[1.]{Static or quasi-static fields \cite{AHLM2009,AmBaoFlem2002,HauKuhPott2005,DassFokKar2005}}
\item[2.]{Homogeneous media \cite{BleCoh1977,BaDuong2000,HauKuhPott2005,AZML2007,AlKressSerr2009,EllVal2009,BaoLinTri2010,AMR2009,DassFokKar2005,MarDevZio2000,LakLouis2008}}
\item[3.]{Point-like sources \cite{AHLM2009,AlKressSerr2009,HeRom1998}}
\end{itemize}
The above assumptions make the problem at hand much more tractable. However, in some situations, they may introduce unrealistic approximations leading to poor imaging of the source. Independence from all three assumptions is rarely found in the literature.

It is usually the case that the source to be identified is embedded in heterogeneous media \cite{AmBaoFlem2002,AlbMonk2006,StefUhl2009,StefUhl2011,DevMarLi2007}. For instance, in the photoacoustic imaging of the brain, it is important to incorporate the sudden change of sound speed across the skull \cite{StefUhl2009,StefUhl2011,YangWang2008}. Hence, we consider the Helmholtz equation with variable coefficients to model the characteristics of the medium. 

We assume the presence of a source function $F$ with its support in a bounded open set $\Omega \subset \mathbb{R}^3$ with boundary $\Gamma$. The radiated field $u$ satisfies the following problem
\begin{eqnarray}
\nabla \cdot a(x) \nabla u(x,k) + k^2 b(x) u(x,k) = - F(x,k) \quad \text{in} \quad \mathbb{R}^3, \label{Eqn:Intro.001} \\
\lim_{r \rightarrow \infty} r \big( \partial_{r} u(x,k) - i k u(x,k) \big) = 0, \label{Eqn:Intro.002}
\end{eqnarray}
where $r = |x|$. We will refer to $k > 0$ as the \emph{frequency} and the limit in (\ref{Eqn:Intro.002}) is known as the Sommerfeld radiation condition. Technical conditions on the coefficients $a(x)$ and $b(x)$ are given in Section \ref{Section:DirectProblem} following the weak formulation of this problem. In the present work, we only consider separable sources of the form
\begin{eqnarray}
F(x,k) = h(k) f(x), \label{Eqn:Intro.003}
\end{eqnarray}
where $h(k) \neq 0$ is a known function. This model appears, for instance, in the formulation of the photoacoustic tomography problem \cite{AZML2007} where $h(k)=i k$. 

The multi-frequency inverse source problem is to identify $f$ from measurements of the radiated field $u$ on $\Gamma$ for all frequencies $k \in \mathcal{K}$ where $\mathcal{K} \subset \mathbb{R}_{+}$ is the set of admissible frequencies. This research topic is motivated by applications in acoustic, elastic and electromagnetic remote sensing where the spatial profile of the source function is frequency-independent \cite{AZML2007,StefUhl2009,StefUhl2011,EllVal2009,BaoLinTri2010,AGKS2011,BLT2011}. The main two goals of the present work are the following:
\begin{itemize}
\item[1.]{Obtain the uniqueness property for the multi-frequency inverse source problem using the weak formulation of the BVP. In contrast with the above three assumptions, we consider multi-frequency wave fields (possibly associated with temporal evolution via the Fourier transform), heterogeneous media, and distributed sources.}
\item[2.]{Formulate a reconstruction algorithm based on a variational characterization of the unknown source and a subspace projection method using measurable data with multi-frequency content.}
\end{itemize}

As shown in Section \ref{Sec:MultiFreqUnique}, the unknown source can be identified from the knowledge of multi-frequency measurements even in the presence of heterogeneous media. Recently, Eller and Valdivia \cite{EllVal2009}, and Bao \textit{et al.} \cite{BaoLinTri2010} and Alves \textit{et al.} \cite{AMR2009} have investigated the same problem for homogeneous media. In \cite{EllVal2009} it was proved that the knowledge of $u(x,k)$ for all $x \in \Gamma$ and all $k \in \mathcal{K}$ is sufficient to recover $f$ uniquely if $\mathcal{K} = \mathbb{R}_{+}$. In fact, their proof only requires data obtained at frequencies coinciding with the eigenvalues of the negative Laplacian in the region $\Omega$. Other uniqueness results were obtained in \cite{AMR2009} and \cite{BaoLinTri2010} by requiring $\mathcal{K}$ to only contain an interval or a sequence with a limit point. They used the fact that plane waves satisfy the Helmholtz equation (homogeneous media) which leads to the Fourier transform of the unknown source $f$ in terms of multi-frequency boundary measurements. Since $f$ has compact support then its Fourier transform is analytic and completely determined by such multi-frequency data.

In our present study for the Helmholtz equation in heterogeneous media, plane waves no longer satisfy the governing equation. Instead of using the analyticity of the Fourier transform of $f$, we managed to directly prove that the radiating wave field $u_{k}$ is analytic with respect to the frequency $k>0$ (see Theorem \ref{Thm:Analyt} below). This leads to the uniqueness result for the inverse source problem for multi-frequency data with an accumulation point. In light of \cite{StefUhl2012} and the use of the Fourier transform in time, one may conjecture that the uniqueness result for the inverse source problem is valid if $h=h(x,k)$ is a non-vanishing known function and analytic with respect to the frequency $k$. Unfortunately, the approach pursued in this paper does not cover such a case.

In order to obtain the multi-frequency uniqueness property and the development of a practical reconstruction method, we review the decomposition of any source $f \in L^{2}(\Omega)$ as the superposition of a non-radiating component $f_{N}$ and a purely-radiating part $f_{P}$ which happens to be the minimum-$L^{2}$-norm solution of the inverse problem. This is done in Sections \ref{Section:IverseProblem} and \ref{Sec:VarCharcSource} using the variational framework. Early results in this direction were obtained by Bleistein and Cohen \cite{BleCoh1977}, Marengo, Devaney and Ziolkowski \cite{MarDevZio2000,MarZiol1999,MarDev2004}, among others, and rigorously established in the variational setting by Albanese and Monk for Maxwell's equations \cite{AlbMonk2006}. In particular, we take great advantage of the concept of purely-radiating sources and their characterization as solutions of the so-called \textit{adjoint} problem \cite{MarDev2004}.

Finally, an algorithm for the reconstruction of the source using multi-frequency data is proposed in Section \ref{Sec:Algorithm}. We assume that the measured data is available for a finite set of increasing frequencies $\mathcal{K} = \{ k_{1}, k_{2}, ..., k_{J}\}$. In this case, the reconstructed source can be characterized as the minimum-norm function in the intersection of $J$ affine subspaces of $L^{2}(\Omega)$. The variational characterization of the source in terms of the adjoint problem (Section \ref{Sec:VarCharcSource}) is employed as a projection method to obtain an approximation to the unknown source. A few numerical examples are presented as well.


\section{The direct source problem} \label{Section:DirectProblem}
In this section we briefly summarize the main properties of the direct source problem for the Helmholtz equation using its weak formulation. The basic tools necessary to prove existence, uniqueness and stability of the solutions are based on the Riesz-Fredholm theory for compact perturbations of a boundedly invertible operator. Such approach is comprehensibly described by Gilbarg and Trudinger \cite{GilTrud2001} and McLean \cite{McLean2000} among others. We also establish the analyticity with respect to the frequency $k$ of a solution to the variable-coefficient BVP (\ref{Eqn:Intro.001})-(\ref{Eqn:Intro.002}). This is accomplished using the analytic Fredholm theory \cite{ColtonKress02,RenRog2004}.

We assume the source $f$ to have its support within a bounded open domain $\Omega \subset \mathbb{R}^3$ whose Lipschitz boundary is denoted by $\Gamma$. Hence, it is possible to reformulate the BVP (\ref{Eqn:Intro.001})-(\ref{Eqn:Intro.002}) in the domain $\Omega$ by replacing the Sommerfeld radiation condition at infinity by the exterior Dirichlet-to-Neumann (DtN) condition on the enclosing surface $\Gamma$. This renders a BVP in $\Omega$ whose solution coincides exactly with the restriction to $\Omega$ of the solution for the original problem (\ref{Eqn:Intro.001})-(\ref{Eqn:Intro.002}). For excellent references concerning the DtN map (also known as the Steklov-Poincar\'{e} operator) see \cite{McLean2000,ColtonKress02,Ihlen1998,NedelecBook2001,MonkBook2003}. The weak formulation of the BVP is well-suited to model highly heterogeneous media, and the DtN formulation allows us to naturally incorporate the boundary $\Gamma$ where the measurements take place.

The following is the weak or variational problem with the Sobolev space $H^{1}(\Omega)$ as the natural choice for the trial and test spaces.
\begin{Problem}[Direct Problem] \label{Prob:VariationalDtN}
Given a source function $f \in L^{2}(\Omega)$, find a function $u \in H^{1}(\Omega)$ satisfying
\begin{eqnarray}
&& B(u,v) = h \la f , v \ra_{L^{2}(\Omega)}, \qquad \text{for all} \quad v \in H^{1}(\Omega) \label{Eqn:VF.001}
\end{eqnarray}
where the sesquilinear form $B :  H^{1}(\Omega) \times  H^{1}(\Omega) \rightarrow \mathbb{C}$ is defined as follows,
\begin{eqnarray}
&& B(u,v) = \la a \nabla u , \nabla v \ra_{L^{2}(\Omega)} - k^2 \la b u , v \ra_{L^{2}(\Omega)} - \la \Lambda \gamma u , \gamma v \ra_{L^{2}(\Gamma)}.  \label{Eqn:VF.002}
\end{eqnarray}
Here $\gamma : H^{1}(\Omega) \rightarrow H^{1/2}(\Gamma)$ is the trace operator, and $\Lambda : H^{1/2}(\Gamma) \rightarrow H^{-1/2}(\Gamma)$ denotes the (exterior) Dirichlet-to-Neumann map which transfers the boundary values of a radiating solution of the Helmholtz equation in $\mathbb{R}^3 \setminus \Omega$ into its normal derivative on $\Gamma$. When needed, we will make explicit reference to the dependence of $u$, $h$ and $\Lambda$ on the frequency $k$ by writing $u=u_{k}$, $h = h_{k}$ and $\Lambda=\Lambda_{k}$.
\end{Problem}
We assume that the coefficients $a(x)$ and $b(x)$ are real and that $a,b \in L^{\infty}(\Omega)$. This ensures the boundedness of the sesquilinear form $B$. In order to work with a uniformly elliptic problem, we also require the existence of a constant $a_{0}$ such that $0 < a_{0} \leq a(x)$. We also assume that $0 < b$, $b^{-1} \in L^{\infty}(\Omega)$, and that the supports of $1-a(x)$ and $1-b(x)$ lie winthin $\Omega$. The above problem is simply the weak formulation of the following BVP,
\begin{eqnarray}
&& \nabla \cdot a \nabla u + k^2 b u = - h f \quad \text{in} \quad \Omega, \label{Eqn:CF.005} \\
&& \partial_{\nu} u - \Lambda u = 0 \quad \text{on} \quad \Gamma , \label{Eqn:CF.006}
\end{eqnarray}
where $\nu$ denotes the outward unit normal vector on $\Gamma$.

The well-posedness of the Direct Problem \ref{Prob:VariationalDtN} follows from the Riesz-Fredholm theory or the Babu\v{s}ka-Lax-Milgram theorem. For details see \cite{Kress01,GilTrud2001,McLean2000,HsiaoWend2008,Ihlen1998}. Hence, we obtain the well-posedness of the Direct Problem \ref{Prob:VariationalDtN}.
\begin{theorem}[Well-Posedness] \label{Thm:WellPosed}
For each $k>0$, the Direct Problem \ref{Prob:VariationalDtN} has a unique stable solution $u \in H^{1}(\Omega)$.
\end{theorem}

Furthermore, to be used in the analysis of the inverse source problem discussed in Section \ref{Sec:MultiFreqUnique}, we shall also establish the analyticity with respect to the frequency $k$ of a solution to the variable-coefficient problem (\ref{Eqn:Intro.001})-(\ref{Eqn:Intro.002}). In order to do so, we modify the use of the Dirichlet-to-Neumann operator $\Lambda : H^{1/2}(\Gamma) \rightarrow H^{-1/2}(\Gamma)$. Consider a smooth simply-connected open region $\tilde{\Omega}$ large enough to contain $\clo{\Omega}$. For instance, $\tilde{\Omega}$ may be a sufficiently large open ball. Let $\tilde{\Gamma}$ be the smooth boundary of $\tilde{\Omega}$. The unique weak solution for the original BVP (\ref{Eqn:Intro.001})-(\ref{Eqn:Intro.002}) satisfies the following variational problem,
\begin{eqnarray}
&& \tilde{B}(u,v) = h \la f , v \ra_{L^{2}(\tilde{\Omega})}, \qquad \text{for all} \quad v \in H^{1}(\tilde{\Omega}) \label{Eqn:VFm.001}
\end{eqnarray}
where the modified sesquilinear form $\tilde{B} :  H^{1}(\tilde{\Omega}) \times  H^{1}(\tilde{\Omega}) \rightarrow \mathbb{C}$ is defined as follows,
\begin{eqnarray}
&& \tilde{B}(u,v) = \la a \nabla u , \nabla v \ra_{L^{2}(\tilde{\Omega})} - k^2 \la b u , v \ra_{L^{2}(\tilde{\Omega})} - \la \tilde{\Lambda} \gamma u , \tilde{\gamma} v \ra_{L^{2}(\tilde{\Gamma})}.  \label{Eqn:VFm.002}
\end{eqnarray}
Here $\gamma : H^{1}(\Omega) \rightarrow H^{1/2}(\Gamma)$ and $\tilde{\gamma} : H^{1}(\tilde{\Omega}) \rightarrow H^{1/2}(\tilde{\Gamma})$ are trace operators, and $\tilde{\Lambda} : H^{1/2}(\Gamma) \rightarrow H^{1/2}(\tilde{\Gamma})$ denotes a \textit{modified} Dirichlet-to-Neumann operator which transfers the boundary values on $\Gamma$ of a radiating solution of the Helmholtz equation in $\mathbb{R}^3 \setminus \Omega$ into its normal derivative on the \textit{larger} surface $\tilde{\Gamma}$. Again, we will make explicit reference to the dependence of $u$, $h$ and $\tilde{\Lambda}$ on the frequency $k$ by writing $u=u_{k}$, $h = h_{k}$ and $\tilde{\Lambda}=\tilde{\Lambda}_{k}$. The regularity and analyticity of $\tilde{\Lambda}_{k}$ is established below.

Now, from the Riesz representation theorem, there exist bounded linear operators $P , Q_{k} : H^{1}(\tilde{\Omega}) \rightarrow H^{1}(\tilde{\Omega})$ and a unique element $F \in H^{1}(\tilde{\Omega})$ such that for all $u,v \in H^{1}(\tilde{\Omega})$,
\begin{eqnarray}
\la P u , v \ra_{H^{1}(\tilde{\Omega})} = \la a \nabla u , \nabla v \ra_{L^{2}(\tilde{\Omega})} + \la u , v \ra_{L^{2}(\tilde{\Omega})}, \label{Eqn:RR.001} \\
\la Q_{k} u , v \ra_{H^{1}(\tilde{\Omega})} = k^2 \la b u , v \ra_{L^{2}(\tilde{\Omega})} + \la u , v \ra_{L^{2}(\tilde{\Omega})} + \la \tilde{\Lambda}_{k} \gamma u , \tilde{\gamma} v \ra_{L^{2}(\tilde{\Omega})}, \label{Eqn:RR.002} \\
\la F , v \ra_{H^{1}(\tilde{\Omega})} = \la f , v \ra_{L^{2}(\tilde{\Omega})}. \label{Eqn:RR.003}
\end{eqnarray}
Then we can re-express the variational problem \ref{Eqn:VFm.001} as the following equation,
\begin{eqnarray}
(P - Q_{k})u_{k} = h_{k} F. \label{Eqn:OP.001}
\end{eqnarray}
Since the coefficients $a , b \in L^{\infty}(\tilde{\Omega})$ and $0 < a_{0} \leq a(x)$, then $P$ is bounded and coercive, and from the Lax-Milgram theorem we conclude that $P$ has a bounded inverse. In addition, since the embeddings of $H^{1}(\tilde{\Omega})$ into $L^2(\tilde{\Omega})$ and of $H^{1/2}(\tilde{\Gamma})$ into $L^{2}(\tilde{\Gamma})$ are compact, then we also get that $Q_{k}$ is a compact operator. That was the reason why we introduced the \textit{larger} surface $\tilde{\Gamma}$ and the modified DtN map $\tilde{\Lambda}$ which enjoys enough regularity. 

The derivation of (\ref{Eqn:OP.001}) yields the possibility to show that the unique solution of (\ref{Eqn:VFm.001}), or equivalently the weak solution of (\ref{Eqn:Intro.001})-(\ref{Eqn:Intro.002}), is in fact analytic with respect to the frequency $k$. First, it is convenient to establish the following lemma.

\begin{lemma} \label{Lemma:AnalytDtN}
Both, the Dirichlet-to-Neumann map $\Lambda_{\xi} : H^{1/2}(\Gamma) \to H^{-1/2}(\Gamma)$ and the \textit{modified} Dirichlet-to-Neumann map $\tilde{\Lambda}_{\xi} : H^{1/2}(\Gamma) \to H^{1/2}(\tilde{\Gamma})$, can be analytically extended as a bounded operators for complex frequencies $\xi \in \mathbb{C}$ in a neighborhood of the positive real line $\mathbb{R}_{+}$.
\end{lemma}

Our proof is provided below at the end of this section. With the aid of this lemma then it follows that the mapping $\xi \mapsto Q_{\xi}$ is analytic in a neighborhood of $\mathbb{R}_{+}$ in the complex plane $\mathbb{C}$. The well-posedness Theorem \ref{Thm:WellPosed} implies that the operator $(P - Q_{\xi})^{-1}$ exists and is bounded for each $\xi \in \mathbb{R}_{+}$. Then we obtain from the analytic Fredholm theorem \cite{RenRog2004} that $(P - Q_{\xi})^{-1}$ is also analytic with respect to $\xi \in \mathbb{R}_{+}$. It follows that the solution of the Direct Problem \ref{Prob:VariationalDtN} is analytic with respect to the frequency $k>0$ because it can be expressed as $u_{k} = h_{k} (P - Q_{k})^{-1} F$. To be used in the analysis of the inverse problem, we state this result in the form of a theorem.
\begin{theorem}[Analyticity] \label{Thm:Analyt}
If $h=h(k)$ is analytic, then the solution $u=u(x,k)$ of the Direct Problem \ref{Prob:VariationalDtN} is analytic with respect to the frequency $k > 0$.
\end{theorem}

Now we proceed to introduce the mathematical tools and notation to prove Lemma \ref{Lemma:AnalytDtN}. The proof is based on boundary integral operators and the analyticity of the fundamental solution for the Helmholtz equation. \emph{The main challenge lies in providing a meaningful extension for the definition of the DtN map $\Lambda_{\xi}$ for complex frequencies $\xi \in \mathbb{C}$}. In order to explain and resolve this difficulty, we make the following definitions concerning regions of interest and associated trace operators.
\begin{eqnarray*}
\text{Exterior domains} \quad & \Omega^{+} = \mathbb{R}^3 \setminus \clo{\Omega} \quad \text{and} \quad \tilde{\Omega}^{+} = \mathbb{R}^3 \setminus \clo{\tilde{\Omega}}. \\
\text{Surfaces} \quad & \Gamma = \partial \Omega \quad \text{and} \quad \tilde{\Gamma} = \partial \tilde{\Omega}. \\
\text{Trace operators} \quad & \gamma^{+} : H_{\rm loc}^{1}(\Omega^{+}) \to H^{1/2}(\Gamma). \\ 
			     & \tilde{\gamma}^{+} : H_{\rm loc}^{1+s}(\tilde{\Omega}^{+}) \to H^{1/2+s}(\tilde{\Gamma}), \quad s \geq 0. \\
\text{Normal derivative operators} \quad & \partial_{\nu}^{+} : H_{\rm loc}^{1}(\Omega^{+}) \to H^{-1/2}(\Gamma). \\
			& \partial_{\tilde{\nu}}^{+} : H_{\rm loc}^{1+s}(\tilde{\Omega}^{+}) \to H^{-1/2+s}(\tilde{\Gamma}), \quad s \geq 0.
\end{eqnarray*}
Here $\nu$ and $\tilde{\nu}$ are the unit normal vector on $\Gamma$ and $\tilde{\Gamma}$ pointing into $\Omega^{+}$ and $\tilde{\Omega}^{+}$, respectively. The normal derivative operators are well-defined to operate on weak solutions of an elliptic equation \cite{McLean2000}.

We consider complex frequencies $\xi \in \mathbb{C}$ with $\mathfrak{Re}~ \xi > 0$. Recall that for a frequency $\xi \in \mathbb{R}_{+}$, the DtN operator $\Lambda_{\xi} : H^{1/2}(\Gamma) \to H^{-1/2}(\Gamma)$ is defined to map boundary values (Dirichlet data) on $\Gamma$ of a radiating field in $\Omega^{+}$ to its normal derivative (Neumann data) on $\Gamma$. If $\xi \in \mathbb{C}$ with $\mathfrak{Im}~ \xi \geq 0$ then the DtN map is well-defined since the unique solvability of the exterior Dirichlet problem in the region $\Omega^{+}$ is guaranteed \cite{McLean2000,ColtonKress02,NedelecBook2001,HsiaoWend2008}. However, if $\mathfrak{Im}~ \xi < 0$, then the solvability of such an exterior problem is no longer certain. In fact, the outgoing fundamental solution for the Helmholtz equation grows exponentially with $r=|x|$ and it does not satisfy the Sommerfeld radiation condition (\ref{Eqn:Intro.002}). So we cannot expect such a BVP to be well-posed for $\mathfrak{Im}~ \xi < 0$. Hence, it is needed to modify the radiation condition to account for such growth.

We review the following exterior problem leading to the definition of the DtN map. Given Dirichlet data $g \in H^{1/2}(\Gamma)$, let $v \in H^{1}_{\rm loc}(\Omega^{+})$ be the weak solution of the following boundary value problem,
\begin{eqnarray}
\Delta v + \xi^2 v = 0 \quad \text{in} \quad \Omega^{+}, \label{Eqn:BVPDtN.001} \\
\gamma^{+} v = g \quad \text{on} \quad \Gamma, \label{Eqn:BVPDtN.002} \\
v(x) = \frac{e^{i \xi |x|}}{4 \pi |x|}\bigg[ v_{\infty}(\hat{x}) + \mathcal{O}(1/|x|) \bigg], \quad \text{as $|x| \to \infty$,} \label{Eqn:BVPDtN.003}
\end{eqnarray}
for some function $v_{\infty} : \mathbb{S}^2 \to \mathbb{C}$ known as the far-field pattern. Here $\hat{x} = x / |x|$ and $\xi \in \mathbb{C}$. The limit (\ref{Eqn:BVPDtN.003}) plays the role of the radiation condition which is satisfied by the fundamental solution for any complex frequency, and it is easily shown to be equivalent to the standard Sommerfeld condition (\ref{Eqn:Intro.002}) whenever $\mathfrak{Im}~ \xi \geq 0$. For the DtN operators to be well-defined, it is necessary to establish the well-posedness of the BVP (\ref{Eqn:BVPDtN.001})-(\ref{Eqn:BVPDtN.003}) for complex frequencies, at least in a neighborhood of the positive real line. Uniqueness follows from the work of Lax and Phillips \cite{LaxPhill1989,LaxPhill1972}. They showed that, aside from a discrete set (no point of accumulation) of complex frequencies with $ \mathfrak{Im}~ \xi < 0$, there are no nontrivial outgoing solutions for the Helmholtz equation with homogeneous Dirichlet boundary condition on $\Gamma$. To prove existence and stability, we make use of surface potential theory \cite{McLean2000,ColtonKress02,NedelecBook2001,HsiaoWend2008}.

Consider the single- and double-layer potentials
\begin{eqnarray}
\eqalign{
(\mathcal{S}_{\xi} \psi)(x) = \int_{\Gamma} \Phi(x,y,\xi) \psi(y) ds(y), \quad  x \in \mathbb{R}^3, \quad \xi \in \mathbb{C},  \\
(\mathcal{D}_{\xi} \psi)(x) = \int_{\Gamma} \frac{\partial \Phi(x,y,\xi)}{\partial \nu(y)} \psi(y) ds(y), \quad x \in \mathbb{R}^3, \quad \xi \in \mathbb{C},}  \label{Eqn.LayerPot}
\end{eqnarray}
and the single- and double-layer boundary integral operators
\begin{eqnarray}
\eqalign{
(S_{\xi} \psi)(x) = 2 (\mathcal{S}_{\xi} \psi)(x), \quad  x \in \Gamma, \quad \xi \in \mathbb{C},  \\
(D_{\xi} \psi)(x) = 2 (\mathcal{D}_{\xi} \psi)(x), \quad x \in \Gamma, \quad \xi \in \mathbb{C},}  \label{Eqn.LayerBIO}
\end{eqnarray}
where $\Phi(x,y,\xi) = e^{i \xi | x - y|} / (4\pi |x-y|)$ is the outgoing fundamental solution for the Helmholtz equation. Inherited from $\Phi(x,y,\xi)$, notice that these potentials and integral operators are analytic with respect to $\xi$. From the smoothness of $\Phi(x,y,\xi)$, we also have that $(S_{\xi} \psi)$ and $(D_{\xi} \psi)$ belong to $H_{\rm loc}^{2}(\tilde{\Omega}^{+})$ since $\text{dist} (\Gamma,\tilde{\Omega}^{+} ) > 0$. Therefore, the potentials (\ref{Eqn.LayerPot}) and integral operators (\ref{Eqn.LayerBIO}) define bounded linear operators satisfying the following regularity properties \cite{McLean2000,ColtonKress02,HsiaoWend2008},
\begin{eqnarray}
\eqalign{
S_{\xi} : H^{1/2}(\Gamma) \to H^{1}(\Gamma), \qquad & D_{\xi} : H^{1/2}(\Gamma) \to H^{1}(\Gamma), \\
\mathcal{S}_{\xi} : H^{1/2}(\Gamma) \to H_{\rm loc}^{1}(\Omega^{+}), \qquad & \mathcal{D}_{\xi} : H^{1/2}(\Gamma) \to H_{\rm loc}^{1}(\Omega^{+}),  \\
\gamma^{+} \mathcal{S}_{\xi} : H^{1/2}(\Gamma) \to H^{1/2}(\Gamma), \qquad & \gamma^{+} \mathcal{D}_{\xi} : H^{1/2}(\Gamma) \to H^{1/2}(\Gamma),  \\
\partial_{\nu}^{+} \mathcal{S}_{\xi} : H^{1/2}(\Gamma) \to H^{-1/2}(\Gamma), \qquad & \partial_{\nu}^{+} \mathcal{D}_{\xi} : H^{1/2}(\Gamma) \to H^{-1/2}(\Gamma), \\
\partial_{\tilde{\nu}}^{+} \mathcal{S}_{\xi} : H^{1/2}(\Gamma) \to H^{1/2}(\tilde{\Gamma}), \qquad & \partial_{\tilde{\nu}}^{+} \mathcal{D}_{\xi} : H^{1/2}(\Gamma) \to H^{1/2}(\tilde{\Gamma}),
} \label{Eqn.RegProp}
\end{eqnarray}
as well as the following jump relations \cite{McLean2000,ColtonKress02,HsiaoWend2008},
\begin{eqnarray}
\eqalign{
2~ (\gamma^{+} \mathcal{S}_{\xi} \psi)(x) = (S_{\xi} \psi)(x), \quad x \in \Gamma, \\
2~ (\gamma^{+} \mathcal{D}_{\xi} \psi)(x) = \psi(x) + (D_{\xi} \psi)(x), \quad x \in \Gamma.} \label{Eqn.jump}
\end{eqnarray}
With these mathematical tools in place, we are ready to obtain the desired proof.

\begin{proof}[Proof of Lemma \ref{Lemma:AnalytDtN}]
The goal is to extend the definition of the DtN map $\Lambda_{\xi}$ for complex frequencies in a neighborhood of the positive real line $\mathbb{R}_{+}$, and make the extension analytic. We accomplish this by mimicking the procedure found in \cite[Section 3.2]{ColtonKress02} to find solutions for exterior problems in the form of a combined single- and double-layer potential,
\begin{eqnarray}
v(x,\xi) = (\mathcal{D}_{\xi} \psi)(x) - i (\mathcal{S}_{\xi} \psi)(x), \quad x \in \Omega^{+}, \quad \xi \in \mathbb{C}. \label{Eqn.Ansatz}
\end{eqnarray}
From the jump relations (\ref{Eqn.jump}) we see that the ansatz (\ref{Eqn.Ansatz}) solves the BVP (\ref{Eqn:BVPDtN.001})-(\ref{Eqn:BVPDtN.003}) provided that the density $\psi \in H^{1/2}(\Gamma)$ solves the following boundary integral equation
\begin{eqnarray}
(I + D_{\xi} - i S_{\xi} ) \psi = 2 g, \quad \text{for $g \in H^{1/2}(\Gamma)$.} \label{Eqn.BIE}
\end{eqnarray}
Since $H^{1}(\Gamma)$ is compactly embedded into $H^{1/2}(\Gamma)$ then both $S_{\xi},D_{\xi} : H^{1/2}(\Gamma) \to H^{1/2}(\Gamma)$ are compact. Thus, the Riesz-Fredholm theory applies. Since the operator $(I + D_{\xi} - i S_{\xi} )$ has been shown to be invertible for each frequency $\xi \in \mathbb{C}$ with $\mathfrak{Im}~ \xi \geq 0$ \cite{ColtonKress02}, then the analytic Fredholm theorem \cite{RenRog2004,ColtonKress02} implies that $(I + D_{\xi} - i S_{\xi} )^{-1} : H^{1/2}(\Gamma) \to H^{1/2}(\Gamma)$ is bounded and analytic with respect to $\xi \in \mathbb{C}$, aside from a discrete set lying in the open lower half-plane. It follows that the ansatz (\ref{Eqn.Ansatz}) solves the BVP (\ref{Eqn:BVPDtN.001})-(\ref{Eqn:BVPDtN.003}) and is analytic in a complex neighborhood of each positive frequency.

Finally, from the regularity properties (\ref{Eqn.RegProp}), we obtain well-defined and bounded DtN operator $\Lambda_{\xi} : H^{1/2}(\Gamma) \to H^{-1/2}(\Gamma)$ and modified DtN operator $\tilde{\Lambda}_{\xi} : H^{1/2}(\Gamma) \to H^{1/2}(\tilde{\Gamma})$ explicitly given by
\begin{eqnarray}
\eqalign{
\Lambda_{\xi} = 2 \partial_{\nu}^{+} ( \mathcal{D}_{\xi} - i \mathcal{S}_{\xi} )(I + D_{\xi} - i S_{\xi} )^{-1}, \\
\tilde{\Lambda}_{\xi} = 2 \partial_{\tilde{\nu}}^{+} ( \mathcal{D}_{\xi} - i \mathcal{S}_{\xi} )(I + D_{\xi} - i S_{\xi} )^{-1}.} \label{Eqn.ExtDtN}
\end{eqnarray}
As desired, the DtN operators are meromorphic with respect to $\xi \in \mathbb{C}$, and analytic in a neighborhood of the positive real line. Moreover, by construction, the definitions in (\ref{Eqn.ExtDtN}) do in fact agree with the respective mappings $g \mapsto \partial_{\nu}^{+} v$ and $g \mapsto \partial_{\tilde{\nu}}^{+} v$ (Dirichlet data into Neumann data) associated with the BVP (\ref{Eqn:BVPDtN.001})-(\ref{Eqn:BVPDtN.003}). This concludes the proof.
\end{proof}


\section{The inverse source problem} \label{Section:IverseProblem}

In this section we state the inverse source problem and review some of its mathematical characteristics such as the lack of uniqueness and the concept of minimum-norm solution. Recall that given a source $f \in L^2(\Omega)$ and a frequency $k >0$ there exists a unique solution $u \in H^{1}(\Omega)$ for the Direct Problem \ref{Prob:VariationalDtN}. When needed, we will make explicit reference to the dependence of $u$ on the frequency $k$ by writing $u=u_{k}$ or $u= u(x,k)$. We consider an admissible set of frequencies denoted by $\mathcal{K} \subset \mathbb{R}_{+}$. With this notation we are ready to define the inverse problem.

\begin{Problem}[Inverse Source Problem] \label{Prob:InvSource}
Let $u_{k} \in H^{1}(\Omega)$ be the solution to the Direct Problem \ref{Prob:VariationalDtN} for given frequency $k \in \mathcal{K}$ and some unknown source $f \in L^{2}(\Omega)$. The inverse source problem is, given the traces $\gamma u_{k}$ on $\Gamma$, find the source $f$.
\end{Problem}

In light of previous results concerning multi-frequency uniqueness \cite{EllVal2009,BaoLinTri2010,AMR2009}, one may ask whether data obtained at finitely many frequencies is sufficient to recover the source $f$ uniquely. By showing the existence of sources that do not radiate at a finite number of distinct frequencies, we answer the above question in the negative.

The precise definition of non-radiating sources for the Helmholtz equation is given as follows.
\begin{definition}[Non-Radiating Source] \label{Def:Nonrad}
A source $f \in L^{2}(\Omega)$ is said to be \emph{non-radiating} at a frequency $k$ if the solution $u_{k} \in H^{1}(\Omega)$ of the Direct Problem \ref{Prob:VariationalDtN} corresponding to this source $f$ is such that $\gamma u_{k}=0$ on $\Gamma$. Let $N(\Omega,k)$ denote the set of all non-radiating sources for the Helmholtz equation at the frequency $k$.
\end{definition}

We now consider the following orthogonal decomposition of $L^{2}(\Omega)$,
\begin{eqnarray}
L^{2}(\Omega) = N(\Omega,k) \oplus N(\Omega,k)^{\perp} \label{Eqn:NUMN.001}
\end{eqnarray}
This direct sum is valid if $N(\Omega,k)$ is a closed subspace of $L^2(\Omega)$. This is the case because $N(\Omega,k)$ is the nullspace of the operator $\gamma \circ S : L^{2}(\Omega) \rightarrow H^{1/2}(\Gamma)$ where $\gamma$ is the trace operator and $S : L^{2}(\Omega) \rightarrow H^{1}(\Omega)$ represents the \emph{solution operator} for the Direct Problem \ref{Prob:VariationalDtN}. From the well-posedness Theorem \ref{Thm:WellPosed} it follows that $S$ is a bounded operator. Since both operators $\gamma$ and $S$ are bounded, then $N(\Omega,k) = \text{null}(\gamma \circ S)$ is a closed subspace of $L^{2}(\Omega)$ and the orthogonal decomposition (\ref{Eqn:NUMN.001}) is well-defined. As a consequence, we have the following definition for purely-radiating sources.
\begin{definition}[Purely-Radiating Source] \label{Def:Purerad}
A source $f \in L^{2}(\Omega)$ is said to be \emph{purely-radiating} at a frequency $k$ if $f \in N(\Omega,k)^{\perp}$.
\end{definition}

It is well-known that set $N(\Omega,k)$ is a non-trivial vector space. This is seen from the definition the following set
\begin{eqnarray}
\fl \mathcal{N}(\Omega,k) = \big\{ g \in L^{2}(\Omega) ~|~ \text{ $ h g = \nabla \cdot a \nabla w + k^2 b w ~$ for some $~ w \in C_{\rm c}^{\infty}(\Omega)$} \big\}. \label{Eqn:NUMN.003}
\end{eqnarray}
It is clear that $\mathcal{N}(\Omega,k) \subset N(\Omega,k)$. Moreover, $\mathcal{N}(\Omega,k)$ is a non-trivial vector space since there are plenty of functions in $C_{\rm c}^{\infty}(\Omega)$ that do not satisfy the Helmholtz equation in $\Omega$. This simple argument shows the existence of non-radiating sources for the Helmholtz equation in heterogeneous media.

The fact that $\mathcal{N}(\Omega,k) \subset N(\Omega,k)$ is a non-trivial space implies that the inverse source problem cannot be solved uniquely given boundary data at a single frequency. Moreover, we may show that data obtained from finitely many frequencies is not sufficient to recover $f$ uniquely. For simplicity, we assume constant coeffcients $a \equiv b \equiv h \equiv 1$. We want to show the existence of non-trivial sources that do not radiate at a finite number of distinct frequencies. This is done in a manner similar to that of the definition of the non-radiating sources found in the set $\mathcal{N}(\Omega,k)$ as given by (\ref{Eqn:NUMN.003}).

Let $L_{1} = \Delta + k_{1}^{2}$ be the Helmholtz operator at a frequency $k_{1}$. Similarly define $L_{2}$ for a frequency $k_{2} \neq k_{1}$. Now let $w \in C_{\rm c}^{\infty}(\Omega)$ be such that $L_{1} L_{2} w \neq 0$, and define $g = L_{1} L_{2} w$, $w_{1} = L_{2} w$ and $w_{2} = L_{1} w$. Notice that all three $g, w_{1}, w_{2} \in C_{\rm c}^{\infty}(\Omega)$. Also notice that the Helmholtz operators commute, that is, $ L_{1} L_{2} w = L_{2} L_{1} w = g$. Hence, the source $g$ is simultaneously non-radiating at the two distinct frequencies $k_{1}$ and $k_{2}$ because it gives rise to the wave fields $w_{1}$ at frequency $k_{1}$ and $w_{2}$ at frequency $k_{2}$ such that $\gamma w_{1} = \gamma w_{2} = 0$. This argument can be extended easily to an arbitrary finite number of frequencies. Hence, concerning the multi-frequency Inverse Source Problem \ref{Prob:InvSource}, boundary data obtained from finitely many frequencies is not sufficient to recover the unknown source $f$.

To conclude this section, we briefly address the concept of minimum-norm solutions for the inverse problem. Notice that if $f$ solves the inverse source problem at a frequency $k$, then $f + g$ for any $g \in N(\Omega,k)$ will solve it as well. Out of the infinitely many solutions to the inverse problem we can select the one that has minimum $L^{2}$-norm. This is a consequence of the well-known best approximation theorems for Hilbert spaces \cite{Deu2001}. Therefore, given an arbitrary source $f \in L^{2}(\Omega)$ there exists a unique decomposition $f = f_{N,k} + f_{P,k}$ with $f_{N,k} \in N(\Omega,k)$ and $f_{P,k} \in N(\Omega,k)^{\perp}$, such that $f_{P,k}$ is the minimum-norm solution for the inverse source problem associated with $f$ at a frequency $k$.


\section{A variational characterization of the unknown source} \label{Sec:VarCharcSource}

We now turn to the characterization of the set $N(\Omega,k)^{\perp}$. The approach follows Marengo and Devaney \cite{MarDev2004}, and Albanese and Monk \cite{AlbMonk2006}. For that reason, we define a variational problem whose solutions will be shown to be dense in $N(\Omega,k)^{\perp}$. This new problem will be called \emph{adjoint} to the original Direct Problem \ref{Prob:VariationalDtN}.

\begin{Problem}[Adjoint Problem] \label{Prob:AdjointProb}
Given boundary data $\eta \in L^{2}(\Gamma)$, find a function $\psi \in H^{1}(\Omega)$ satisfying
\begin{eqnarray}
&& A(\psi,\phi) = \la \eta , \gamma \phi \ra_{L^{2}(\Gamma)}, \qquad \text{for all} \quad \phi \in H^{1}(\Omega), \label{Eqn:NUMN.005}
\end{eqnarray}
where the sesquilinear form $A :  H^{1}(\Omega) \times  H^{1}(\Omega) \rightarrow \mathbb{C}$ is defined as follows,
\begin{eqnarray}
&& A(\psi,\phi) = \la a \nabla \psi , \nabla \phi \ra_{L^{2}(\Omega)} - k^2 \la b \psi , \phi \ra_{L^{2}(\Omega)} - \la \Lambda^{*} \gamma \psi , \gamma \phi \ra_{L^{2}(\Gamma)}.  \label{Eqn:NUMN.006}
\end{eqnarray}
Here $\gamma : H^{1}(\Omega) \rightarrow H^{1/2}(\Gamma)$ is the trace operator, and $\Lambda^{*} : H^{1/2}(\Gamma) \rightarrow H^{-1/2}(\Gamma)$ denotes the adjoint of the Dirichlet-to-Neumann operator. Notice also that $\psi$ depends on the frequency $k$, so we may write $\psi = \psi_{k}$.
\end{Problem}

Let us denote the set of solutions of the Adjoint Problem \ref{Prob:AdjointProb} by
\begin{eqnarray}
&& \mathcal{P}(\Omega,k) = \big\{ \psi \in H^{1}(\Omega) \text{ a solution of Problem \ref{Prob:AdjointProb} for some $\eta \in L^{2}(\Gamma)$} \nonumber \\
&& \qquad \qquad \qquad \text{ and fixed frequency $k$} \big\}, \label{Eqn:NUMN.007} \\
&& P(\Omega,k) = \clo{\mathcal{P}(\Omega,k)} \quad \text{(closure in the $L^{2}(\Omega)$-norm)}. \label{Eqn:NUMN.008}
\end{eqnarray}
In order for $\mathcal{P}(\Omega,k)$ to be well-defined it is necessary to show that the Adjoint Problem \ref{Prob:AdjointProb} is well-posed. This we do in the form of a lemma.
\begin{lemma} \label{Lemma:AdjProbWellPosed}
The Adjoint Problem \ref{Prob:AdjointProb} is well-posed.
\end{lemma}

\begin{proof}
From the Fredholm theory, one obtains that the well-posedness of the Direct Problem \ref{Prob:VariationalDtN} implies the well-posedness of the Adjoint Problem \ref{Prob:AdjointProb} and vice versa. For details see \cite{McLean2000}.
\end{proof}

As done in \cite{AlbMonk2006} for Maxwell equations, the Adjoint Problem \ref{Prob:AdjointProb} can be used to obtain a variational characterization of the unknown source $f$. This is done by developing a link between $f$ and the boundary measurements $\gamma u$ on the surface $\Gamma$. Let $\psi \in H^{1}(\Omega)$ be a solution to the Adjoint Problem \ref{Prob:AdjointProb} for some $\eta \in L^{2}(\Gamma)$. Then
\begin{eqnarray}
&& A(\psi , \phi) = \la \eta , \gamma \phi \ra_{L^{2}(\Gamma)}, \qquad \text{for all} \quad \phi \in H^{1}(\Omega), \label{Eqn:RFT.009}
\end{eqnarray}
Denote by $u \in H^{1}(\Omega)$, the solution to the original Direct Problem \ref{Prob:VariationalDtN} for the unknown source $f$. Then we also have,
\begin{eqnarray}
&& B(u , v) = h \la f , v \ra_{L^{2}(\Omega)}, \qquad \text{for all} \quad v \in H^{1}(\Omega).  \label{Eqn:RFT.011}
\end{eqnarray}

Now, we let $\phi = u$ and $v = \psi$, and the combination of (\ref{Eqn:RFT.009}) and (\ref{Eqn:RFT.011}) renders,
\begin{eqnarray}
h \la f , \psi \ra_{L^{2}(\Omega)} = \la \gamma u , \eta \ra_{L^{2}(\Gamma)}, \qquad \text{for all} \quad \psi \in \mathcal{P}(\Omega,k). \label{Eqn:RFT.015}
\end{eqnarray}
The above variational characterization of the unknown source $f$ can be employed to obtain the projection of $f$ on the space $P(\Omega,k)$. This is important as we shall prove in Theorem \ref{Thm:NPperp} below that $P(\Omega,k)$ is in fact the set of all purely-radiating sources $N(\Omega,k)^{\perp}$. In other words, expression (\ref{Eqn:RFT.015}) may be used to compute the minimum-norm solution of the inverse source problem.

The definition of the set $P(\Omega,k)$ as a closed subspace of $L^{2}(\Omega)$ allows us to establish another orthogonal decomposition
\begin{eqnarray}
L^{2}(\Omega) = P(\Omega,k) \oplus P(\Omega,k)^{\perp}. \label{Eqn:NUMN.010}
\end{eqnarray}
Now we proceed to characterize the space $N(\Omega,k)^{\perp}$ by proving that $N(\Omega,k)^{\perp} = P(\Omega,k)$ or equivalently that $N(\Omega,k) = P(\Omega,k)^{\perp}$
\begin{theorem} \label{Thm:NPperp}
A function $f \in L^{2}(\Omega)$ belongs to $P(\Omega,k)^{\perp}$ if and only if it belongs to $N(\Omega,k)$.
\end{theorem}

\begin{proof}
Let $f \in L^{2}(\Omega)$ be arbitrary. The proof relies on the variational characterization (\ref{Eqn:RFT.015}) where $u \in H^{1}(\Omega)$ is the unique solution to the Direct Problem \ref{Prob:VariationalDtN}.

Now assume that $f \in P(\Omega,k)^{\perp}$. In (\ref{Eqn:RFT.015}) let $\eta = \gamma u \in H^{1/2}(\Gamma) \subset L^{2}(\Gamma)$ and $\psi \in \mathcal{P}(\Omega,k)$ be the associated solution of the Adjoint Problem \ref{Prob:AdjointProb}. Then it follows that $\gamma u = 0$ which means that $f \in N(\Omega,k)$.

Conversely, if $f \in N(\Omega,k)$ then $\gamma u = 0$ by Definition \ref{Def:Nonrad}. It follows from (\ref{Eqn:RFT.015}) that $\la f , \psi \ra_{L^{2}(\Omega)}=0$ for all $\psi \in \mathcal{P}(\Omega,k)$. Hence $f \in P(\Omega,k)^{\perp}$ since $\mathcal{P}(\Omega,k)$ is dense in $P(\Omega,k)$.
\end{proof}

In summary, we conclude from Theorem \ref{Thm:NPperp} that the orthogonal decompositions (\ref{Eqn:NUMN.001}) and (\ref{Eqn:NUMN.010}) are the same because $N(\Omega,k) = P(\Omega,k)^{\perp}$ and $P(\Omega,k) = N(\Omega,k)^{\perp}$. In addition, the variational characterization (\ref{Eqn:RFT.015}) can be employed to obtain the projection of the unknown source $f$ on the space $P(\Omega,k)$ and this projection coincides with the minimum-norm solution of the inverse source problem. Finally, we wish to mention that the Adjoint Problem \ref{Prob:AdjointProb} is just the variational formulation of the following boundary value problem,
\begin{eqnarray}
&& \nabla \cdot a \nabla \psi + k^2 b \psi = 0 \quad \text{in} \quad \Omega, \nonumber \\
&& \partial_{\nu} \psi - \Lambda^{*} \psi = \eta \quad \text{on} \quad \Gamma. \nonumber
\end{eqnarray}
This means that purely-radiating sources for the Helmholtz equation are themselves (approximate) weak solutions to the Helmholtz equation. This is the subject of several papers \cite{AlbMonk2006,MarDevZio2000,MarZiol1999,MarDev2004}.


\section{Multi-frequency uniqueness} \label{Sec:MultiFreqUnique}

The variational characterization (\ref{Eqn:RFT.015}) has great practical as well as theoretical value. It will be used in Section \ref{Sec:Algorithm} to derive a reconstruction algorithm. However, we will employ it here to establish the uniqueness property for the Inverse Source Problem \ref{Prob:InvSource} for measurements of the wave field $u_{k}$ on the surface $\Gamma$ at many frequencies $k \in \mathcal{K}$. The goal is to prove that the unknown source $f$ can be identified uniquely if $\mathcal{K}$ contains a sequence with a limit point. This was done by Bao \emph{et al.} \cite{BaoLinTri2010} for the Helmholtz equation in homogeneous media. From the fact that plane waves satisfy the Helmholtz equation, they were able to use Green's second identity to obtain the Fourier transform of the unknown source $f$ in terms of multi-frequency boundary measurements. Since $f$ has compact support then its Fourier transform is analytic and completely determined by data with an accumulation point.

In our present study for the Helmholtz equation in heterogeneous media, plane waves no longer satisfy the governing equation. Instead of using the analyticity of the Fourier transform of $f$, we already managed to directly prove that the radiating wave field $u_{k}$ is analytic with respect to the frequency $k>0$ (see Theorem \ref{Thm:Analyt}). As seen below, this will have the same effect on establishing uniqueness for the inverse source problem for multi-frequency data with an accumulation point, even in the case of heterogeneous media.

Recall the variational characterization (\ref{Eqn:RFT.015}) for the source $f \in L^{2}(\Omega)$,
\begin{eqnarray}
h_{k} \la f , \psi_{k} \ra_{L^{2}(\Omega)} = \la \gamma u_{k} , \eta_{k} \ra_{L^{2}(\Gamma)}, \qquad \text{for all} \quad \psi_{k} \in \mathcal{P}(\Omega,k). \nonumber
\end{eqnarray}
Here $\psi_{k} \in H^{1}(\Omega)$ represents any arbitrary solution of the Adjoint Problem \ref{Prob:AdjointProb} for some boundary data $\eta_{k} \in L^{2}(\Gamma)$ and $k>0$. In fact, we may carefully choose a frequency $k$ and the field $\psi_{k}$ so that they weakly solve the following eigenvalue problem,
\begin{eqnarray}
&& \nabla \cdot a \nabla \psi = - \lambda b \psi  \quad \text{in} \quad \Omega, \label{Eqn.SLP1} \\
&& \psi = 0 \quad \text{on} \quad \Gamma. \label{Eqn.SLP2}
\end{eqnarray}
This is a particular case of the well-known regular Sturm-Liouville problem \cite{RenRog2004} from which we learn that there is a countable set of eigenpairs $(\psi_{j},\lambda_{j})_{j=1}^{\infty}$ such that $\lambda_{j}>0$ and the eigenfunctions form an orthonormal basis of the Hilbert space $L^{2}(\Omega,b)$ with respect to the inner product $\la f , g \ra_{L^{2}(\Omega,b)} = \la f , g b \ra_{L^{2}(\Omega)}$. Recall that $0 < b , b^{-1} \in L^{\infty}(\Omega)$.

The above convenient choice of Dirichlet eigenfunctions combined with expression (\ref{Eqn:RFT.015}) yields the unique determination of the unknown source $f$ just as in the reconstruction algorithm devised in \cite{EllVal2009} for homogeneous media. In fact, if $f \in L^2(\Omega)$, then $f b^{-1} \in L^2(\Omega,b)$. Then we have a generalized Fourier expansion,
\begin{eqnarray}
f b^{-1} = \sum_{j=1}^{\infty} \alpha_{j} \psi_{j}, \label{Eqn.FourierExp1}
\end{eqnarray}
where $ \alpha_{j} = \la f b^{-1} , \psi_{j} \ra_{L^{2}(\Omega,b)} = \la f b^{-1}, \psi_{j} b \ra_{L^{2}(\Omega)} = \la f, \psi_{j} \ra_{L^{2}(\Omega)} = h_{j}^{-1} \la \gamma u_{j} , \eta_{j} \ra_{L^{2}(\Gamma)}$. Here $\eta_{j} = \partial_{\nu} \psi_{j}$, $u_{j}$ is the solution of the Direct Problem \ref{Prob:VariationalDtN} for a frequency $k_{j} = \lambda_{j}^{1/2}$ and $h_{j} = h(k_{j})$. At this point, we are ready to state and prove the uniqueness result for the multi-frequency inverse source problem with heterogeneous media.

\begin{lemma} \label{Lemma:UniqueInv}
Let $\mathcal{K} = \{ \lambda_{j}^{1/2} \}_{j=1}^{\infty}$ where $\lambda_{j}$ are the eigenvalues of the Sturm-Liouville problem (\ref{Eqn.SLP1})-(\ref{Eqn.SLP2}). Suppose that $f^{(1)}, f^{(2)} \in L^{2}(\Omega)$ are two sources such that their radiated waves (solutions of the Direct Problem \ref{Prob:VariationalDtN}) coincide on the surface $\Gamma$ for all frequencies in $\mathcal{K}$. Then $f^{(1)}=f^{(2)}$.
\end{lemma}

\begin{proof}
Let $u_{j}^{(1)}$ and $u_{j}^{(2)}$ be the wave fields radiated at frequencies $k_{j} = \lambda_{j}^{1/2} \in \mathcal{K}$ by $f^{(1)}$ and $f^{(2)}$, respectively.
Set $f = f^{(1)} - f^{(2)}$ and $u_{j} = u_{j}^{(1)}-u_{j}^{(2)}$, and notice that by the linearity of the Helmholtz equation, then $u_{j}$ is the solution of the Direct Problem \ref{Prob:VariationalDtN} for the source $f$. The goal is to show that $f=0$. Since the wave fields $u_{j}^{(1)}$ and $u_{j}^{(2)}$ coincide on the surface $\Gamma$, then we have that $\gamma u_{j} = 0$ for all $j \in \mathbb{N}$.  Now, from the generalized Fourier expansion (\ref{Eqn.FourierExp1}) of the function $f b^{-1}$ we obtain that all of its Fourier coefficients are $\alpha_{j} = h_{j}^{-1} \la \gamma u_{j} , \eta_{j} \ra_{L^{2}(\Gamma)} = 0$. Hence $f = 0$ as desired.
\end{proof}

The above uniqueness result can be improved by exploiting the analyticity of a solution $u_{k}$ of the Direct Problem \ref{Prob:VariationalDtN} with respect to the frequency $k > 0$ as stated in Theorem \ref{Thm:Analyt}.

\begin{theorem} \label{Thm:UniqueInv}
Let $\mathcal{K} \subset \mathbb{R}_{+}$ have an accumulation point and $h=h(k)$ be analytic. Suppose that $f^{(1)}, f^{(2)} \in L^{2}(\Omega)$ are two sources such that their radiated waves (solutions of the Direct Problem \ref{Prob:VariationalDtN}) coincide on the surface $\Gamma$ for all frequencies $k \in \mathcal{K}$. Then $f^{(1)}=f^{(2)}$.
\end{theorem}

\begin{proof}
The result follows from the fact that the solution $u_{k}$ of the Direct Problem \ref{Prob:VariationalDtN} is analytic with respect to $k \in \mathbb{R}_{+}$ as stated in Theorem \ref{Thm:Analyt}. Therefore, the knowledge of the traces $\gamma u_{k}$ for $k \in \mathcal{K}$ determines $\gamma u_{k}$ in a neighborhood of the accumulation point of $\mathcal{K}$ \cite{Rudin1987}. By the standard analytic continuation principle, then $\gamma u_{k}$ is uniquely determined for all $k > 0$ and the result follows from Lemma \ref{Lemma:UniqueInv}.
\end{proof}


\section{Reconstruction method} \label{Sec:Algorithm}

Now we propose a reconstruction algorithm based on a simple projection method in the Hilbert space setting \cite{Deu2001}. This method fits perfectly with the variational characterization (\ref{Eqn:RFT.015}) of the unknown source. We assume that the measured data is available for a finite set of frequencies $\mathcal{K} = \{ k_{1} < k_{2} < ... < k_{J}\}$. In this case, the reconstructed source can be characterized as the minimum-norm function from the intersection of $J$ affine subspaces of $L^{2}(\Omega)$. The variational characterization of the source is employed in connection with projection methods to obtain an approximation to the unknown source.

Recall the variational characterization (\ref{Eqn:RFT.015}) for the source $f \in L^{2}(\Omega)$,
\begin{eqnarray}
h_{k} \la f , \psi_{k} \ra_{L^{2}(\Omega)} = \la \gamma u_{k} , \eta_{k} \ra_{L^{2}(\Gamma)}, \qquad \text{for all} \quad \psi_{k} \in \mathcal{P}(\Omega,k). \nonumber
\end{eqnarray}
Here $\psi_{k} \in H^{1}(\Omega)$ represents any arbitrary solution of the Adjoint Problem \ref{Prob:AdjointProb} for some boundary data $\eta_{k} \in L^{2}(\Gamma)$ and $k>0$. For each fixed $k >0$, one may be tempted to sample the space $L^{2}(\Gamma)$ with many choices of $\eta_{k}$ in order to obtain a good number of associated test functions $\psi_{k}$. Instead, we consider the follow decomposition
\begin{eqnarray}
L^{2}(\Gamma) = \text{span} \{ \gamma u_{k} \} \oplus \text{span} \{\gamma u_{k} \}^{\perp}, \label{Eqn.DecompEta}
\end{eqnarray}
which reveals the fact that data from $\text{span} \{\gamma u_{k} \}^{\perp}$ will not contribute with new information to recover $f$ from the use of the variational characterization (\ref{Eqn:RFT.015}). Hence, we simply choose $\eta_{k} = \gamma u_{k}$ which is the measured data. The goal is to find $f \in L^{2}(\Omega)$ that satisfies
\begin{eqnarray}
\la f , \psi_{j} \ra_{L^{2}(\Omega)} = h_{j}^{-1} \| \gamma u_{j} \|_{L^{2}(\Gamma)}^{2}, \quad j=1,2,...,J, \label{Eqn.System01}
\end{eqnarray}
where $\gamma u_{j}$ are the measurements on the surface $\Gamma$ and $\psi_{j}$ is the (computed) solution of the Adjoint Problem \ref{Prob:AdjointProb} for frequencies $k_{j}$, and $h_{j} = h(k_{j})$ for $j=1,2,...,J$. Setting,
\begin{eqnarray}
\mathcal{H}_{j} = \bigg\{ g \in L^{2}(\Omega) ~|~ \la g , \psi_{j} \ra_{L^{2}(\Omega)} = h_{j}^{-1} \| \gamma u_{j} \|_{L^{2}(\Gamma)}^{2}  \bigg\}, \quad \mbox{and} \quad
\mathcal{H} = \bigcap_{j=1}^{J} \mathcal{H}_{j}, \nonumber
\end{eqnarray}
we want to find a function in $\mathcal{H}$ that solves the system (\ref{Eqn.System01}). Out of the infinitely many solutions, we seek the one with minimum $L^{2}(\Omega)$-norm which we denote by $f_{\rm min}$. It follows from the best approximation theory \cite{Deu2001} that
\begin{eqnarray}
f_{\rm min} = \sum_{i=1}^{J} \alpha_{i} \psi_{i} \label{Eqn.MinMulti}
\end{eqnarray}
where the scalars $\alpha_{i}$ satisfy the normal equations
\begin{eqnarray}
\sum_{i=1}^{J} \alpha_{i} \la \psi_{i} , \psi_{j} \ra_{L^{2}(\Omega)} = h_{j}^{-1} \| \gamma u_{j} \|_{L^{2}(\Gamma)}^{2}, \quad j=1,2,...,J,  \label{Eqn.NormalEqns}
\end{eqnarray}
provided that the functions $\psi_{j}$ are linearly independent. The latter can be shown as follows. Consider,
\begin{eqnarray}
\sum_{j=1}^{J} \beta_{j} \psi_{j} = 0. \nonumber
\end{eqnarray}
Apply the gradient operator, multiply by the variable coefficient $a(x)$ and perform the $L^{2}(\Omega)$ inner-product against $\nabla \phi$ where $\phi \in H_{0}^{1}(\Omega)$ to obtain
\begin{eqnarray}
\sum_{j=1}^{J} \beta_{j} \la a \nabla \psi_{j} , \nabla \phi \ra_{L^{2}(\Omega)} = \sum_{j=1}^{J} \beta_{j} k_{j}^{2} \la b \psi_{j} , \phi \ra_{L^{2}(\Omega)} = 0, \quad \text{for all $\phi \in H_{0}^{1}(\Omega)$.} \nonumber
\end{eqnarray}
Since $H_{0}^{1}(\Omega)$ is dense in $L^{2}(\Omega)$ and $b(x) > 0$, we obtain that $\sum_{j=1}^{J} \beta_{j} ~ k_{j}^{2} ~\psi_{j} = 0$. The above steps can be applied as many times as desired to obtain $\sum_{j=1}^{J} \beta_{j} ~ k_{j}^{2 m} ~\psi_{j} = 0$, for all $m \in \mathbb{N}$. Since $0 < k_{1} < k_{2} < ... < k_{J}$, we may divide by $k_{J}^{2m}$ and the limit as $m \to \infty$ yields that $\beta_{J} = 0$. Similarly, we obtain that $\beta_{j} = 0$ for all $j=1,2,...,J$. This establishes the linear independence of the adjoint fields $\psi_{j}$.

\subsection{Numerical examples} \label{Sec:Examples}
Let $\Omega = B(0,2)$ the open ball with center at the origin and radius equal to $2$, and  $h(k) \equiv 1$. We will only illustrate spherically symmetric examples, so we work with the radial variable $r = |x|$. We consider two sources
\begin{eqnarray*}
\fl f^{(1)}(r) = 
\cases{
1 & \text{if $r \leq \frac{1}{2}$,} \\
0 & \text{otherwise,}
} \quad \text{and} \quad 
f^{(2)}(r) = 
\cases{
1 + \cos (2 \pi r) & \text{if $r \leq \frac{1}{2}$,} \\
0 & \text{otherwise.}
}
\end{eqnarray*}
The characteristics of the medium are modelled by the coefficients,
\begin{eqnarray*}
\fl a(r) = 
\cases{
\text{$\frac{1}{2}$} & \text{if $\frac{3}{4} \leq r \leq 1$,} \\
0 & \text{otherwise,}
} \quad \text{and} \quad 
b(r) = 
\cases{
2 & \text{if $\frac{3}{4} \leq r \leq 1$,} \\
0 & \text{otherwise.}
}
\end{eqnarray*}

In order to illustrate the performance of the reconstruction method, a sequence of numerical
tests was made for increasingly larger sets of frequencies $\mathcal{K} = \{ k_{j} = j ~:~ j=1,...,J\}$. Using the proposed algorithm (\ref{Eqn.MinMulti})-(\ref{Eqn.NormalEqns}) and numerically generated measurements, the reconstructed sources $f_{J}^{(1)}$ and $f_{J}^{(2)}$ were obtained as a function of the parameter $J$. The relative error between the true and recovered sources is computed as follows,
\begin{eqnarray}
\epsilon_{J}^{(i)} = \frac{\| f^{(i)} - f_{J}^{(i)} \|_{L^{2}(\Omega)} }{\| f^{(i)} \|_{L^{2}(\Omega)}}, \qquad \text{for $i=1,2$}, \label{Eq.error}
\end{eqnarray}
and the results are displayed in Table \ref{Table1}.

\begin{table}[h!] 
\caption{\label{Table1} Relative error (\ref{Eq.error}).}  
\begin{indented} \item[]
\begin{tabular}{lllll}
\br
$J$& ~ & $\epsilon_{J}^{(1)}$ & ~ & $\epsilon_{J}^{(2)}$ \\ 
\mr
10 & ~ & 4.14e-1 & ~ & 1.30e-1  \\ 
20 & ~ & 2.93e-1 & ~ & 1.92e-2  \\ 
30 & ~ & 2.42e-1 & ~ & 6.27e-3  \\ 
40 & ~ & 2.12e-1 & ~ & 4.25e-3  \\ 
50 & ~ & 1.88e-1 & ~ & 2.74e-3  \\
60 & ~ & 1.73e-1 & ~ & 2.34e-3  \\
\br
\end{tabular} 
\end{indented}
\end{table}

We also display the true and reconstructed sources for $J=40$ in Figure \ref{Fig.Fig1}. From these numerical experiments, it is clear that the proposed reconstruction method performs better when the source is smooth. This is typically the case in the general framework of harmonic approximation theory where convergence rates are directly related to the smoothness of the function being approximated. It is seen from Table \ref{Table1} that the convergence rate for the sources $f^{(1)}$ and $f^{(2)}$ is about $\sim 0.5$ and $ \sim 2.5$, respectively. This  coincides roughly with smoothness degree (in the Sobolev sense) of these two sources. We also see in Figure \ref{Fig.Fig1}, that the reconstruction of a discontinuous sources, such as $f^{(1)}$, is polluted by the well-known Gibbs phenomenon. 

\begin{figure}[h!] 
\begin{indented} \item[]
\hspace{-0.05 \textwidth}
 \includegraphics[width=0.48 \textwidth]{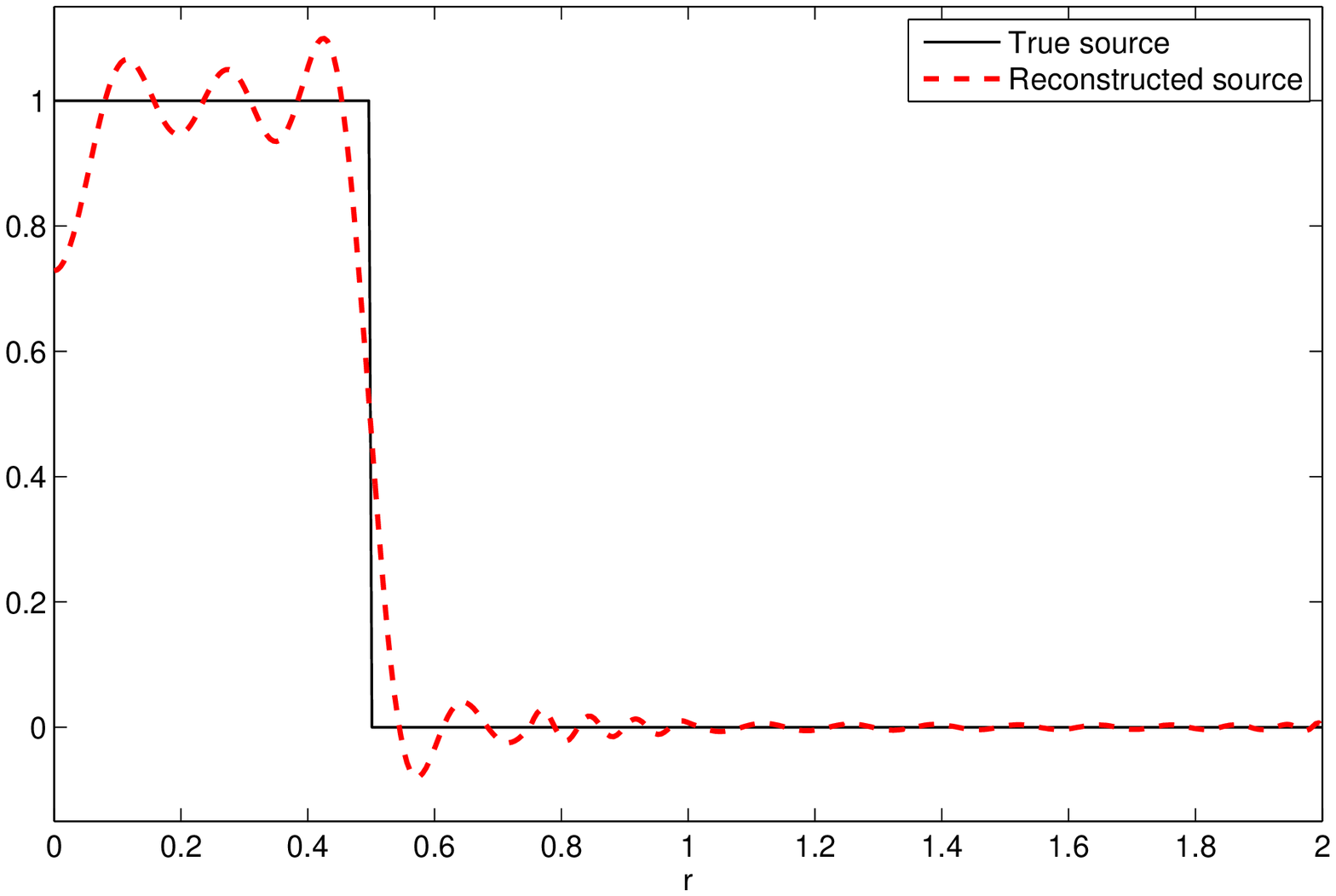} 
\hspace{-0.05 \textwidth}
 \includegraphics[width=0.48 \textwidth]{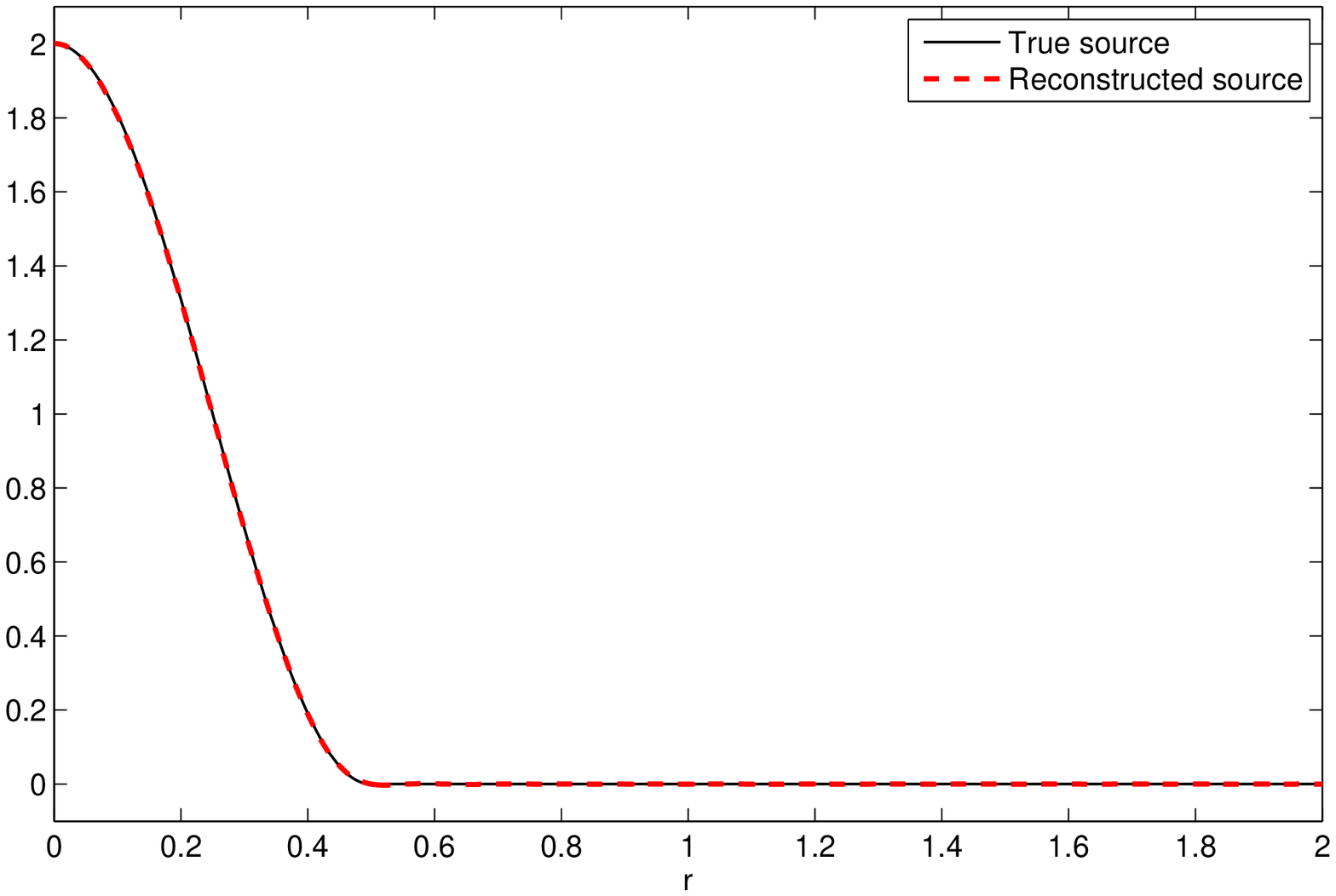}
\caption{\label{Fig.Fig1} Comparison between true sources $f^{(1)}$ and $f^{(2)}$, and corresponding reconstructed sources $f^{(1)}_{J}$ and $f^{(2)}_{J}$ for $J = 40$.}
\end{indented}
\end{figure}

\subsection{Future work} \label{Sec:Future}
We have presented a few numerical examples to illustrate the results from the proposed reconstruction algorithm (\ref{Eqn.MinMulti})-(\ref{Eqn.NormalEqns}). These two examples only consider spherically symmetric sources and corresponding radiating fields. The three-dimensional numerical implementation and a stability analysis for the proposed algorithm are the subject of our current work. Our efforts are directed towards the approximation of the adjoint fields $\psi_{j}$ using the finite element method \cite{Ihlen1998}. The goal is to numerically solve the Adjoint Problem \ref{Prob:AdjointProb} using accurate approximations for the adjoint DtN operator $\Lambda^{*}$. Fortunately, the DtN map $\Lambda$ has become a powerful tool to numerically handle BVPs in unbounded domains \cite{Ihlen1998}. Approximations to this operator are commonly known as absorbing or nonreflecting boundary conditions designed to ensure the \emph{outgoing} behavior of the radiating fields. Analogous approximations can be derived for the adjoint DtN map $\Lambda^{*}$ associated with \emph{incoming} wave fields.

We are also interested in the stability and convergence analysis for the system (\ref{Eqn.NormalEqns}). From physical considerations, one would naturally expect the propose algorithm to resolve finer features of the true source as $k_{J} \to \infty$ while keeping the difference between consecutive frequencies reasonably small. It is also important to consider strategies such as \cite{AdcHan2012} for the resolution of the Gibbs phenomenon seen in the reconstruction of discontinuous sources. As soon as meaningful results are obtained, they will be reported in a forthcoming paper.


\section*{Acknowledgments} \label{Sec:Acknowledgements}
The authors gratefully acknowledge the support provided by the CHIRP grant from the College of Physical and Mathematical Sciences at Brigham Young University. The authors also wish to
thank the anonymous referees for their most constructive reports.


\section*{References}

\bibliographystyle{unsrt}
\bibliography{Biblio}

\begin{thebibliography}{10}

\bibitem{Balanis2005}
C.~A. Balanis.
\newblock {\em Antenna Theory - Analysis and Design}.
\newblock John Wiley \& Sons, Hoboken, NJ, 3rd edition, 2005.

\bibitem{AngKirKlein1991}
T.~Angell, A.~Kirsch, and R.~Kleinman.
\newblock Antenna control and optimization using generalized characteristic
  modes.
\newblock {\em Proc. IEEE}, 79:1559--1568, 1991.

\bibitem{AHLM2009}
A.~Babda, F.~Hassen, J.~Leblond, and M.~Mahjoub.
\newblock Sources recovery from boundary data : \textsc{A} model related to
  electroencephalography.
\newblock {\em Math. Comp. Modelling}, 49:2213--2223, 2009.

\bibitem{AlKressSerr2009}
C.~Alves, R.~Kress, and P.~Serranho.
\newblock Iterative and range test methods for an inverse source problem for
  acoustic waves.
\newblock {\em Inverse Problems}, 25:1--17, 2009.

\bibitem{HeRom1998}
S.~He and V.~Romanov.
\newblock Identification of dipole sources in a bounded domain for
  \textsc{M}axwell's equations.
\newblock {\em Wave Motion}, 28:25--40, 1998.

\bibitem{AmBaoFlem2002}
H.~Ammari, G.~Bao, and J.~Fleming.
\newblock Inverse source problem for \textsc{M}axwell's equation in
  magnetoencephalography.
\newblock {\em SIAM J. Appl. Math.}, 62:1369--1382, 2002.

\bibitem{AlbMonk2006}
R.~Albanese and P.~Monk.
\newblock The inverse source problem for \textsc{M}axwell's equations.
\newblock {\em Inverse Problems}, 22:1023--1035, 2006.

\bibitem{AZML2007}
M.~Anastasio, J.~Zhang, D.~Modgil, and P.~La Riviere.
\newblock Application of inverse source concepts to photoacoustic tomography.
\newblock {\em Inverse Problems}, 23:S21--S35, 2007.

\bibitem{StefUhl2009}
P.~Stefanov and G.~Uhlmann.
\newblock Thermoacoustic tomography with variable sound speed.
\newblock {\em Inverse Problems}, 25:075011, 2009.

\bibitem{StefUhl2011}
P.~Stefanov and G.~Uhlmann.
\newblock Thermoacoustic tomography arising in brain imaging.
\newblock {\em Inverse Problems}, 27:045004, 2011.

\bibitem{ArrRev1999}
S.~Arridge.
\newblock Optical tomography in medical imaging.
\newblock {\em Inverse Problems}, 15:R41--R93, 1999.

\bibitem{JirMc2007}
V.~Jirsa and AR. McIntosh.
\newblock {\em Handbook of Brain Connectivity}.
\newblock Springer-Verlag, New York, 2007.

\bibitem{LuKaufman2003}
Z.~Lu and L.~Kaufman.
\newblock {\em Magnetic Source Imaging of the Human Brain}.
\newblock Lawrence Erlbaum Associates, Inc., 2003.

\bibitem{Amm2009}
H.~Ammari, editor.
\newblock {\em Mathematical modeling in Biomedical Imaging I}, volume 1983 of
  {\em Lecture Notes in Mathematics}.
\newblock Springer-Verlag, Berlin, 2009.

\bibitem{Amm2012}
H.~Ammari, editor.
\newblock {\em Mathematical modeling in Biomedical Imaging II}, volume 2035 of
  {\em Lecture Notes in Mathematics}.
\newblock Springer-Verlag, Berlin, 2012.

\bibitem{Isakov1998}
V.~Isakov.
\newblock {\em Inverse Problems for Partial Differential Equations}, volume 127
  of {\em Applied Mathematical Sciences}.
\newblock Springer, New York, 1998.

\bibitem{Isakov1990}
V.~Isakov.
\newblock {\em Inverse Source Problems}, volume~34 of {\em Math. Surveys and
  Monographs}.
\newblock American Math. Society, Providence, RI, 1990.

\bibitem{HauKuhPott2005}
K.~Hauer, L.~Kuhn, and R.~Potthast.
\newblock On uniqueness and non-uniqueness for current reconstruction from
  magnetic fields.
\newblock {\em Inverse Problems}, 21:955--967, 2005.

\bibitem{DassFokKar2005}
G.~Dassios, A.~Fokas, and F.~Kariotou.
\newblock On the non-uniqueness of the inverse \textsc{MEG} problem.
\newblock {\em Inverse Problems}, 21:L1--L5, 2005.

\bibitem{BleCoh1977}
N.~Bleistein and J.~Cohen.
\newblock Nonuniqueness in the inverse source problem in acoustics and
  electromagnetics.
\newblock {\em J. Math. Phys.}, 18:194--201, 1977.

\bibitem{MarDevZio2000}
E.~Marengo, A.~Devaney, and R.~Ziolkowski.
\newblock Inverse source problem and minimum-energy sources.
\newblock {\em J. Opt. Soc. Am. A}, 17:34--45, 2000.

\bibitem{BaDuong2000}
A.~El Badia and T.~Ha-Duong.
\newblock An inverse source problem in potential analysis.
\newblock {\em Inverse Problems}, 16:651--663, 2000.

\bibitem{EllVal2009}
M.~Eller and N.~Valdivia.
\newblock Acoustic source identification using multiple frequency information.
\newblock {\em Inverse Problems}, 25:115005, 2009.

\bibitem{BaoLinTri2010}
G.~Bao, J.~Lin, and F.~Triki.
\newblock A multi-frequency inverse source problem.
\newblock {\em J. Diff. Equations}, 249:3443--3465, 2010.

\bibitem{AMR2009}
C.~Alves, N.~Martins, and N.~Roberty.
\newblock Full identification of acoustic sources with multiple frequencies and
  boundary measurements.
\newblock {\em Inv. Problems and Imaging}, 3:275--294, 2009.

\bibitem{LakLouis2008}
A.~Lakhal and A.K. Louis.
\newblock Locating radiating sources for \textsc{M}axwell's equations using the
  approximate inverse.
\newblock {\em Inverse Problems}, 24:045020, 2008.

\bibitem{DevMarLi2007}
A.~Devaney, E.~Marengo, and M.~Li.
\newblock Inverse source problem in inhomogeneous background media.
\newblock {\em SIAM J. Appl. Math.}, 67:1353--1378, 2007.

\bibitem{YangWang2008}
X.~Yang and L.~Wang.
\newblock Monkey brain cortex imaging by photoacoustic tomography.
\newblock {\em J. Biomed. Opt.}, 13:044009, 2008.

\bibitem{AGKS2011}
H.~Ammari, J.~Garnier, H.~Kang, and K.~S{\o}lna, editors.
\newblock {\em Mathematical and Statistical Methods for Imaging}, volume 548 of
  {\em Contemporary Mathematics}.
\newblock American Math. Society, Providence, RI, 2011.

\bibitem{BLT2011}
G.~Bao, J.~Lin, and F.~Triki.
\newblock An inverse source problem with multiple frequency data.
\newblock {\em C. R. Acad. Sci. Paris, Ser. I}, 349:855--859, 2011.

\bibitem{StefUhl2012}
P.~Stefanov and G.~Uhlmann.
\newblock Recovery of a source term or a speed with one measurement and
  applications.
\newblock {\em Trans. Amer. Math. Soc.}, To appear, 2012.

\bibitem{MarZiol1999}
E.~Marengo and R.~Ziolkowski.
\newblock On the radiating and nonradiating components of scalar,
  electromagnetic, and weak gravitational sources.
\newblock {\em Phys. Rev. Lett.}, 83:3345--3349, 1999.

\bibitem{MarDev2004}
E.~Marengo and A.~Devaney.
\newblock Nonradiating sources with connections to the adjoint problem.
\newblock {\em Phys. Rev. E}, 70:037601, 2004.

\bibitem{GilTrud2001}
D.~Gilbarg and N.~Trudinger.
\newblock {\em Elliptic Partial Differential Equations of Second Order}.
\newblock Springer, Berlin, 2001.

\bibitem{McLean2000}
W.~McLean.
\newblock {\em Strongly Elliptic Systems and Boundary Integral Equations}.
\newblock Cambridge Univ. Press, Cambridge, 2000.

\bibitem{ColtonKress02}
D.~Colton and R.~Kress.
\newblock {\em Inverse Acoustic and Electromagnetic Scattering Theory},
  volume~93 of {\em Applied Mathematical Sciences}.
\newblock Springer, Berlin, 2nd edition, 1998.

\bibitem{RenRog2004}
M.~Renardy and R.~Rogers.
\newblock {\em An introduction to partial differential equations}, volume~13 of
  {\em Texts in Appl. Math.}
\newblock Springer-Verlag, New York, 2nd edition, 2004.

\bibitem{Ihlen1998}
F.~Ihlenburg.
\newblock {\em Finite Element Analysis of Acoustic Scattering}, volume 132 of
  {\em Applied Mathematical Sciences}.
\newblock Springer-Verlag, New York, 1998.

\bibitem{NedelecBook2001}
J.~N\'{e}d\'{e}lec.
\newblock {\em Acoustic and Electromagnetic Equations : Integral
  Representations for Harmonic Problems}, volume 144 of {\em Applied
  Mathematical Sciences}.
\newblock Springer, New York, 2001.

\bibitem{MonkBook2003}
P.~Monk.
\newblock {\em Finite Element Methods for \textsc{M}axwell's Equations}.
\newblock Oxford University Press, New York, 2003.

\bibitem{Kress01}
R.~Kress.
\newblock {\em Linear Integral Equations}, volume~82 of {\em Applied
  Mathematical Sciences}.
\newblock Springer, New York, 2nd edition, 1999.

\bibitem{HsiaoWend2008}
G.C. Hsiao and W.L. Wendland.
\newblock {\em Boundary Integral Equations}, volume 164 of {\em Applied
  Mathematical Sciences}.
\newblock Springer-Verlag, Berlin, 2008.

\bibitem{LaxPhill1989}
P.D. Lax and R.S. Phillips.
\newblock {\em Scattering Theory}, volume~26 of {\em Pure and Applied
  Mathematics}.
\newblock Academic Press, Inc., San Diego, CA, revised edition, 1989.

\bibitem{LaxPhill1972}
P.D. Lax and R.S. Phillips.
\newblock On the scattering frequencies of the \textsc{L}aplace operator for
  exterior domains.
\newblock {\em Comm. Pure and Appl. Math.}, 25:85--101, 1972.

\bibitem{Deu2001}
F.~Deutsch.
\newblock {\em Best Approximation in Inner Product Spaces}.
\newblock Springer-Verlag, New York, 2001.

\bibitem{Rudin1987}
W.~Rudin.
\newblock {\em Real and Complex Analysis}.
\newblock McGraw-Hill, New York, 3rd edition, 1987.

\bibitem{AdcHan2012}
B.~Adcock and A.~Hansen.
\newblock Stable reconstructions in \textsc{H}ilbert spaces and the resolution
  of the \textsc{G}ibbs phenomenon.
\newblock {\em Appl. Comput. Harmon. Anal.}, 32:357--388, 2012.

\end{thebibliography}

\end{document}